\documentclass[11pt]{article}
\usepackage{amsfonts,amssymb,amstext}
\usepackage{amsmath}
\usepackage{amstext}
\usepackage{amsthm}

\makeatletter

\newcommand{\beq}{\begin{equation}}
\newcommand{\eeq}{\end{equation}}
\newcommand{\bea}{\begin{eqnarray}}
\newcommand{\eea}{\end{eqnarray}}
\newcommand{\beaa}{\begin{eqnarray*}}
\newcommand{\eeaa}{\end{eqnarray*}}

\newcommand{\n}{\noindent}
\newcommand{\q}{\quad}

\newtheorem{thm}{Theorem}[section]

\newtheorem{rem}{Remark}[section]

\newcommand{\g}{\gamma}

\newcommand{\G}{\Gamma}

\newcommand{\al}{\alpha}
\newcommand{\la}{\lambda}
\newcommand{\f}{\infty}

\newcommand{\cd}{\cdot}
\newcommand{\si}{\sigma}

\newcommand{\be}{\beta}
\newcommand{\Th}{\Theta}

\newcommand{\itf}{\int_{-\infty}^{+\infty}}

\newcommand{\bsl}{\baselineskip}

\newcommand{\ConvD}{\overset{d}{\longrightarrow}}
\newcommand{\ConvP}{\overset{P}{\longrightarrow}}

\numberwithin{equation}{section}

\begin{document}

\title{On the robustness to small trends of parameter estimation for
continuous-time stationary models with memory}

\author{M. S. Ginovyan\footnote{The research of M. S. Ginovyan was partially supported by National Science Foundation Grant \#DMS-1309009 at Boston University.} \, and A. A. Sahakyan}
\date{}
\maketitle

\begin{abstract}
The paper deals with a question of robustness of inferences,
carried out on a continuous-time stationary process contaminated
by a small trend, to this departure from stationarity.
We show that a smoothed periodogram approach to parameter
estimation is highly robust to the presence of a small trend in the model.
The obtained result is a continuous version of that of  Hede and  Dai
(Journal of Time Series Analysis, 17, 141-150, 1996) for discrete
time processes.
\vskip2mm
\noindent
{\bf Keywords.} Trend; robust inference; short, intermediate
and long memory; smoothed periodogram; parameter estimation.
\end{abstract}


\section{Introduction}
\label{INT}
Much of statistical inferences about unknown spectral parameters is
concerned with the discrete-time stationary models, in which case it
is assumed that the model is centered, or has a constant mean
(see, e.g., Beran et al. \cite{Ber2},  Dzhaparidze \cite{Dz},
Giraitis et al. \cite{GKS}, Taniguchi and Kakizawa \cite{TK},
and references therein).
In this paper we are concerned with the robustness of inferences,
carried out on a continuous-time stationary process contaminated
by a small trend, to this departure from stationarity.

Specifically, let $\{Y(t),\, t\in \mathbb{R}\}$ be a a zero mean stationary
process with spectral density $f(\la,\theta)$, where
$\theta:=(\theta_1, \ldots, \theta_p)\in\Th\subset\mathbb{R}^p$
is an {\em unknown} vector parameter.
We want to make inferences about $\theta$ in the case where the actual
observed data are in the contaminated form:
\beq
\label{i1}
X(t)=Y(t) + M(t), \quad 0\le t\le T,
\eeq
where $M(t)$ is a deterministic trend.

We assume that the trend $M(t)$ is small, that is, we consider the situation
in which the major trend is removed from the model and a certain component
that remains in the model has only minor effect.
In these cases standard inferences can be carried on the basis
of the stationary model $Y(t)$, and we are interested in question whether
the conclusions are robust against this kind of departure from the stationary model.

A sufficiently developed inferential theory is now available for a
continuous-time stationary model $Y(t)$.
For instance, in Anh et al. \cite{ALS,ALS2},
Avram et al. \cite{AvLS}, Casas and Gao \cite{CG}, Gao \cite{Go},
Gao et al. \cite{GAHT,GAH}, Leonenko and Sakhno \cite{LS}
were obtained sufficient conditions ensuring consistency and asymptotic
normality of various statistical estimators of $\theta$, including quasi
maximum likelihood (Whittle) and minimum contrast estimators of $\theta$,
constructed on the basis of a finite realization
${\bf Y}_T:=\{Y(t)$, $0\le t\le T\}$ of the process $Y(t)$.

In this paper we show that under some conditions on the process
$Y(t)$ and the deterministic trend $M(t)$ the above asymptotic
properties of Whittle and minimum contrast estimators remains valid
for the model $X(t)$, that is, the estimating procedure is
relatively robust against replacing the stationary model $Y(t)$
by the non-stationary model $X(t)$ of the form (\ref{i1}).

We will be concerned with this question for models which may exhibit
long memory, short memory or intermediate memory.

Throughout the paper the letters $C$ and $c$ are used to denote positive
constants, the values of which can vary from line to line.

The paper is structured as follows.
In Section \ref{Model} we describe the statistical model. Section \ref{AR}
contains the approach and the main result of the paper - Theorem \ref{th1}.
Section \ref{Proof} is devoted to the proof of Theorem \ref{th1}.

\section{The model: long memory, short memory and intermediate memory processes}
\label{Model}
Let $\{Y(t), \, t\in \mathbb{R}\}$ be a centered, real-valued, continuous-time
second-order stationary process with covariance function $r(t)$ $(t\in \mathbb{R})$,
possessing a spectral density $f(\la)$ ($\la\in \mathbb{R}$), that is, ${E}[|Y(t)|^2]<\f$,
${E}[Y(t)]=0$, $r(t)={E}[Y(t+u)Y(u)]$ ($u, t\in \mathbb{R}$), and $r(t)$ and
$f(\la)$ are connected by the Fourier integral:
\beq
\label{ad1}
r(t)=\int_{\mathbb{R}}e^{i\la t}f(\la)d\la, \q t\in \mathbb{R}.
\eeq
There are several possible definitions of the notion of "memory"
of a stationary process, and they are not necessarily identical
(see, e.g., Beran et al. \cite{Ber2}, Gao \cite{Go}, Giraitis et al.
\cite{GKS}, Heyde and Dai \cite{HD}, Taniguchi and Kakizawa \cite{TK}).
In this paper, we define the memory concept basing on the integrability
property of covariance function $r(t)$, and depending on the memory
structure we will distinguish the following types of stationary models:
(a) short memory or short-range dependent,
(b) long memory or long-range dependent,
(c) intermediate memory or anti-persistent.

We will say that the process $Y(t)$ displays {\sl short memory (SM)}
or {\sl short-range dependence (SRD)} if the covariance function $r(t)$
is integrable: $r\in L^1(\mathbb{R})$ and $\itf r(t)dt\neq 0$.
In this case the spectral density $f(\lambda)$ is bounded away from
zero and infinity at frequency $\lambda=0$, that is, $0< f(0) <\f.$

A typical continuous-time short memory model example is the
stationary continuous-time autoregressive moving average (CARMA)
process whose spectral density is a rational function
(see, e.g., Brockwell \cite{Br1}).

Much of statistical inference is concerned with the short memory
stationary models.
However, data in many fields of science (economics, finance, hydrology,
telecommunications, etc.) is well modeled by a stationary process with
{\sl unbounded} or {\sl vanishing} at the origin spectral density
(see, e.g., Beran et al. \cite{Ber2}, Casas and Gao \cite{CG}, Gao \cite{Go},
Tsai and Chan \cite{TC}, and references therein).

The process $Y(t)$ is said to be {\sl anti-persistent} or exhibits
{\sl intermediate memory (IM)} if the covariance function $r(t)$
is integrable: $r\in L^1(\mathbb{R})$ and $\itf r(t)dt = 0$.
In this case the spectral density $f(\lambda)$ vanishes at frequency zero:
$f(0) =0.$

We will say that the process $Y(t)$ displays {\sl long memory (LM)} or
{\sl long-range dependence (LRD)} if the covariance function $r(t)$
is not integrable: $r\notin L^1(\mathbb{R})$.
In this case the spectral density $f(\lambda)$ has a pole at frequency zero,
that is, it is unbounded at the origin.

The memory property of a stationary process can also be characterized
by the behavior of spectral density $f(\la)$ in the neighborhood of zero,
or by the behavior of covariance function $r(t)$ at infinity (see, e.g.,
Beran et al. \cite{Ber2}, Section 1.3.4).

An example of a continuous-time model that displays the above defined
memory structures
is the continuous-time autoregressive fractionally integrated moving-average
(CARFIMA) process (see, e.g., Chambers \cite{Ch}, Tsai and Chan \cite{TC}).


In the continuous context, a basic process which has commonly been
used to model LRD is fractional Brownian motion (fBm) $B_H(t)$
with Hurst index $H$. This is a Gaussian process with stationary
increments and spectral density of the form
\beq
 \label{lr3}
 f(\la)\sim c\,|\la|^{1-2H}, \q c>0, \ \,\, 1/2<H<1,
 \eeq
as $\la\to 0$, and covariance function:
\beq
 \label{clr3}
r(t)\sim c\,t^{2H-2}, \,\,  1/2<H<1,
 \eeq
as $t\to \f$, where the symbol $"\sim"$ indicates that the ratio of left- and right-hand sides tends to 1.
Notice that the form (\ref{lr3}) can be understood
in a limiting sense, since the fBm $B_H$ is a nonstationary process
(see, e.g., Solo \cite{So}, Gao et al. \cite{GAHT}).


A proper stationary model in lieu of fBm is the fractional
Riesz-Bessel motion (fRBm), introduced in Anh et al. \cite{AAR},
and then extensively discussed in a number of papers (see,
e.g., Anh et al. \cite{ALM}, Gao et al. \cite{GAHT},
Leonenko and Sakhno \cite{LS}, and references therein).
The fRBm is defined to be a continuous-time Gaussian stationary process
with spectral density of the form
\begin{equation}
\label{rb1}
f(\la)=\frac{{c}}{|\la|^{2u}{(1+\la^2)^{v}}}, \q \la\in\mathbb{R},
\, 0<c<\f, \, 0<u<1/2, \,\ v>0.
\end{equation}

\n Observe that the spectral density (\ref{rb1}) behaves
as $O(|\la|^{-2u})$ as $|\la|\to 0$ and as $O(|\la|^{-2(u+v)})$
as $|\la|\to \f$.
Thus, under the conditions $0<u<1/2$, $v>0$ and $u+v>1/2$
the function $f(\la)$ in (\ref{rb1}) is well-defined for both
$|\la|\to 0$ and  $|\la|\to \f$ due to the presence of the component
$(1+\la^2)^{-v}$, which is the Fourier transform of the Bessel potential.
Note that in the spacial case $0<u<1/2$, $v>1/2$ the condition $u+v>1/2$
holds automatically.
The exponent $u$ determines the LRD, while the exponent $v$
indicates the second-order intermittency of the fRBm
(see, e.g.,  Anh et al. \cite{ALM} and Gao et al. \cite{GAHT}).

Comparing (\ref{lr3}) and (\ref{rb1}), we observe that the
spectral density of fBm is the limiting case as $v\to 0$ that of
fRBm with Hurst index $H=u+1/2.$
\n Thus, the form (\ref{rb1}) means that fRBm may exhibit both
LRD and second-order intermittency.

The next result, which was proved in Ginovyan and Sahakyan \cite{GS5},
gives an asymptotic formula for covariance function of an fRBm:
Let $f(\la)$ be as in (\ref{rb1}) with $0<u<1/2$ and $v>1/2$,
and let $r(t):\,=\hat f(t)$ be the Fourier transform of $f(\la)$, then
\beq
\label{s04}
r(t)=C t^{2u-1}\sin(\pi \al)\G(1-2u)
\cdot\left(1+o(1)\right)\q {\rm as} \q t\to\f.
\eeq

\section{The approach and results}
\label{AR}
The basic approach in estimating unknown spectral parameters,
originated by Whittle \cite{Wh1}, is based on the smoothed periodogram
analysis on a frequency domain, involving approximation of the likelihood
function and asymptotic distributions of empirical spectral functionals.

The Whittle estimation procedure, originally devised for discrete-time
short memory stationary processes, has played a major role in the parametric
estimation in the frequency domain, and was the focus of interest
of many statisticians.
Their aim was to weaken the conditions needed to guarantee the
validity of the Whittle approximation for short memory models,
to find analogues for long and intermediate memory models,
and to show that the Whittle estimator is asymptotically equivalent
to exact maximum likelihood estimator (see, e. g., Dahlhaus \cite{D},
Dzhaparidze \cite{Dz}, Fox and Taqqu \cite{FT2}, Giraitis and Surgailis
\cite{GirS}, Giraitis et al. \cite{GKS}, and references therein).
In particular, it was shown that for Gaussian and linear stationary models the
Whittle approach leads to consistent and asymptotically normal estimators
with the standard rate of convergence under short, intermediate and long
memory assumptions.


Continuous versions of Whittle estimation procedure have been
considered, for example, in Anh et al. \cite{ALS,ALS2},
Avram et al. \cite{AvLS}, Casas and Gao \cite{CG}, Gao \cite{Go},
Gao et al. \cite{GAHT,GAH}, Leonenko and Sakhno \cite{LS}.

The procedure of estimation of a parameter $\theta$ involved in
the spectral density $f(\la,\theta)$ of the model, based on a finite
realization ${\bf Y}_T:=\{Y(t)$, $0\le t\le T\}$ of the centered
stationary process $Y(t)$, is to choose the estimator $\hat\theta_{W}$
to minimize the weighted Whittle functional:
\beq
\label{pe9}
U_{w,T}(\theta): =\frac1{4\pi}\itf\left[\log f(\la, \theta) +
\frac{I_{T,Y}(\la)}{f(\la, \theta)}\right]\cd w(\la) \, d\la,
\eeq
where
\begin{equation} \label{e4}
I_{T,Y}(\la) =\frac1{2\pi T}\left|\int_0^Te^{i\la t}Y(t)\,dt\right|^2
\end{equation}
is the "continuous" periodogram of $Y(t)$,
and $w(\la)$ is an even weight function (that is, $w(-\la)=w(\la)$, $w(\la)\ge0$,
and $w(\la)\in L^1(\mathbb{R})$) for which the integral in (\ref{pe9})
is well defined. The choice of an appropriate weight function depends
on the specific form of the spectral density (see, e.g., Anh et al. \cite{ALS2}).
An example of common used weight function is $w(\la)=1/(1+\la^2)$.

Thus, the  Whittle estimator $\hat\theta_{W}$ with weight function
$w(\la)$ is defined to be a solution of the following estimating equation
\bea
\label{re07}
\itf\left[I_{T,Y}(\la)-{f(\la, \theta)} \right]
\frac{\partial}{\partial\theta}f^{-1}(\la, \theta)\cd w(\la) d\la=0,
\eea
obtained by differentiating under the integral sign in (\ref{pe9}).

The asymptotic properties of the Whittle estimator $\hat\theta_{W}$ then
can be obtained using the standard Taylor expansion methods based on
the following smoothed periodogram convergence results:
\bea
\label{re009}
\itf g(\la, \theta)I_{T,Y}(\la)d\la\ConvP \itf g(\la, \theta){f(\la, \theta)}d\la
\quad {\rm as}\quad T\to\f,
\eea
and
\bea
\label{re09}
T^{1/2}\itf g(\la, \theta)\left[I_{T,Y}(\la)-{f(\la, \theta)} \right]d\la
\ConvD \xi\sim N(0,\si^2)\quad {\rm as}\quad T\to\f
\eea
with
\beq\label{eq:sigma}
\si^2 = 16\pi^3\int_{-\infty}^\infty f^2(\la, \theta) g^2(\la, \theta)d\la,
\eeq
where $g(\la,\theta)=\frac{\partial}{\partial\theta}f^{-1}(\la, \theta)w(\la)$,
$\, I_{T,Y}(\la)$ is the periodogram of $Y(t)$ given by (\ref{e4}), $N(0,\si^2)$
denotes the normal law with mean zero and variance $\si^2$, and $\ConvD$ and $\ConvP$
stand for convergence in distribution and in probability, respectively.

Using this approach, statistical properties of Whittle minimum contrast
estimators for conti\-nuous-time stationary processes were studied
in Anh et al. \cite{ALS}, Avram et al. \cite{AvLS}, Casas and Gao
\cite{CG}, Gao \cite{Go}, Gao et al. \cite{GAHT,GAH}, Leonenko and Sakhno
\cite{LS}.
In particular, consistency and asymptotic normality of Whittle
minimum contrast estimator $\hat\theta_{W}$ was established for
some classes of stationary models, including the fractional Riesz-Bessel
motion model, specified by spectral density $f(\la)=f(\la; \theta)$ given by
(\ref{rb1}) with $\theta=(u, v, c)$.

In our analysis we will use a general even integrable smoothing function
$g(\la; \theta)$ rather than the specific form
$g(\la,\theta)=\frac{\partial}{\partial\theta}f^{-1}(\la, \theta)w(\la)$
which is suggested by the Whittle procedure in (\ref{re07}).
The general estimator $\hat\theta_G$ of $\theta$ is then obtained as a
solution of the estimating equation
\bea
\label{re8}
\itf\left[I_{T,Y}(\la)-{f(\la, \theta)} \right]g(\la, \theta)d\la=0.
\eea
Then the asymptotic properties of the estimator $\hat\theta_G$ can be
obtained from smoothed periodogram convergence results of type (\ref{re009})
and(\ref{re09}) with general smoothing function $g(\la; \theta)$.

Notice that in the continuous context the basic tool for derivation of
limit theorems for empirical spectral functionals of the form
\bea
\label{re88}
J(g):=\itf g(\la, \theta)I_{T,Y}(\la)d\la
\eea
is a central limit theorem for Toeplitz type quadratic functionals of
stationary processes (see Ginovyan \cite{G-3,G-4}, Ginovyan and
Sahakyan \cite{GS2} for Gaussian processes, and Bai et al. \cite{AvLS,BGT2}
for linear processes).

It can be shown that the standard Taylor expansion methods based on the
smoothed periodogram convergence results of type (\ref{re009}) and (\ref{re09})
with a general smoothing function $g(\la; \theta)$ and with the contaminated
periodogram $I_{T,X}(\la)$ instead of $I_{T,Y}(\la)$, lead consistent and
asymptotically normally distributed estimators of $\theta$.
We will not pursue this matter here (the details will be reported elsewhere),
however, notice that in the special case of Whittle procedure, where
$g(\la; \theta)=\frac{\partial}{\partial}\frac{1}{f(\la, \theta)}\cd w(\la)$
the results of Anh et al. \cite{ALS}, Avram et al. \cite{AvLS}, Casas and Gao
\cite{CG}, Gao \cite{Go}, Gao et al. \cite{GAHT,GAH}, Leonenko and Sakhno
\cite{LS} concerning consistency and asymptotic normality of the Whittle
minimum contrast estimators constructed on the basis of the periodogram
$I_{T,Y}(\la)$, continue to hold without change for estimators calculated
on the basis of the contaminated periodogram $I_{T,X}(\la)$,
under appropriate assumptions imposed on the model $Y(t)$ on the smoothing
function $g(\la, \theta)$ and on the trend $M(t)$.

In Theorem \ref{th1} that follows we show that
a small trend of the form $|M(t)|\leq C|t|^{-\beta}$ 
does not effect the asymptotic properties (\ref{re009}) and (\ref{re09}) of the
smoothed periodogram, and hence, the asymptotic properties of the estimator
$\hat\theta_G$, even if $I_{T,Y}(\la)$ is replaced by the contaminated
periodogram $I_{T,X}(\la)$.


\begin{thm}
\label{th1}
Suppose that the stationary mean zero process $\{Y(t), \, t\in \mathbb{R}\}$
in ({\ref{i1}}) is such that the asymptotic relations (\ref{re009}) and (\ref{re09})
are satisfied with general even integrable smoothing function $g(\la)$
and $\si^2$ as in
(\ref{eq:sigma}).
If the trend $M(t)$ and the Fourier transform $a(t):=\widehat g(t)$
of smoothing function $g(\la )$ are such that $M(t)$ is locally integrable
on ${\mathbb R}$ and
\beq\label{re151}
|M(t)|\leq C|t|^{-\beta},\qquad |a(t)|\leq C|t|^{-\g}, \quad
t\in\mathbb{R},\quad 2\beta+\g>\frac 32,
\eeq
with some constants $C > 0$, $\g>0$ and  $\be>1/4$, then
\bea
\label{re9}
T^{1/2}\itf g(\la, \theta)\left[I_{T,X}(\la)-I_{T,Y}(\la) \right]d\la
\ConvP 0 \quad {\rm as}\quad T\to\f,
\eea
and hence the asymptotic relations (\ref{re009}) and (\ref{re09}) are satisfied
with $I_{T,Y}(\la)$ replaced by the contaminated periodogram $I_{T,X}(\la)$,
provided that one of the following conditions holds:
\begin{itemize}
\item[(i)]
the process $Y(t)$ has SM or IM, that is, the covariance function $r(t)$ of   $Y(t)$
satisfies $\ r\in L^1(\mathbb{R})$, and  $\be+\g>1$,
\item[(ii)]
the process $Y(t)$ has LM with 
covariance function $r(t)$ satisfying
\beq\label{re99}
|r(t)| \le C|t|^{-\al}, \quad  t\in\mathbb{R},\quad \al+\g\ge\frac32
\eeq
with some constants  $C>0$, $0<\al\leq 1$,  and $\al+2\be>1$ if $\be<1<\g$.
\end{itemize}
\end{thm}


\begin{rem}
{\rm
It is easy to check that the statement of Theorem \ref{th1} holds,
in particular, if the parameters $\al$, $\be$ and $\g$ satisfy the following
conditions:

 in the case  {\it (i)}: $\be>1/2,\ \g\geq1/2$,

 in the case {\it (ii)}: $\al\geq3/4, \ \be>3/8,\ \g\geq3/4$.}
\end{rem}
\begin{rem}
{\rm 
The discrete version of Theorem \ref{th1}
(with additional conditions $\g=1$ in the case {\it (i)}, and $\g>1,\ \al<1/2$
in the case {\it (ii)}), was proved by Heyde and Dai \cite{HD}
(see also Taniguchi and Kakizawa \cite{TK}, Theorems 6.4.1 and 6.4.2).
Using the same arguments applied in the proof of Theorem \ref{th1}
one can prove that the complete discrete  analog of Theorem \ref{th1} is also true.}

\end{rem}
\begin{rem}
{\rm Convergence results of type (\ref{re009}) and (\ref{re09}) holds under
broad circumstances of SM, IM and LM . For detailed conditions see, for example,
Avram et al. \cite{AvLS}, Ginovyan \cite{G-2} -- \cite{G-9}, Ginovyan and
Sahakyan \cite{GS2}, and Leonenko and Sakhno \cite{LS}.}
\end{rem}
\begin{rem}
{\rm The conditions imposed on the Fourier transform of generating function
$g(t)$ in (\ref{re151}) and on the covariance function $r(t)$ in (\ref{re99})
ensure central limit theorem for empirical functionals of Gaussian and linear
long memory processes.
This can be seen from the considerations of Theorem 5 of Ginovyan and Sahakyan
\cite{GS2} (for Gaussian processes), and
Theorem 2.1 and Corollary 2.1 of Bai et al. \cite{BGT2} (for linear processes).}
\end{rem}

\section{Proof of the main result}
\label{Proof}
\begin{proof}[Proof of Theorem \ref{th1}]
In view of (\ref{i1}) and (\ref{e4}) we can write
\bea\label{re12}
\nonumber
I_{T,X}(\la) - I_{T,Y}(\la)&=&\frac1{2\pi T}\left(\left|\int_0^Te^{i\la t}X(t)\,dt\right|^2 -
\left|\int_0^Te^{i\la t}Y(t)\,dt\right|^2\right)\\
\nonumber
&=&
\frac1{2\pi T}\left(\left|\int_0^Te^{i\la t}[Y(t)+M(t)]\,dt\right|^2 -
\left|\int_0^Te^{i\la t}Y(t)\,dt\right|^2\right)\\
\nonumber
&=&\frac1{2\pi T}\int_0^T\int_0^Te^{i\la (t-s)} \left[Y(t)M(s)+Y(s)M(t)+M(t)M(s)\right]\,dtds
\eea
and
\bea\label{re13}
\nonumber
&&\itf g(\la, \theta)\left[I_{T,X}(\la)-I_{T,Y}(\la) \right]d\la\\
&&=\frac1T\int_0^T\int_0^T\left[Y(t)M(s)+Y(s)M(t)+M(t)M(s)\right]a(t-s)\,dtds.
\eea

Thus, to complete the proof it is enough to prove that under the conditions
of the theorem we have
\beq\label{re15}
T^{-1/2}\int_0^T\int_0^TM(t)M(s)a(t-s)\,dtds \to 0\quad {\rm as}\quad T\to\f
\eeq
and
\bea\label{re14}
T^{-1/2}\int_0^T\int_0^TY(t)M(s)a(t-s)\,dtds \ConvP 0\quad {\rm as}\quad T\to\f.
\eea

{ \sl Proof of}  (\ref{re15}).  For $T>2$ we set
\bea\label{re153}
I(T)&:=&
 \int_0^T\int_0^T\left|M(t)M(s)a(t-s)\right|\,dtds \\
&=&
\int_0^{1}\int_0^{2}+\int_0^{1}\int_{2}^T +\int_{1}^T\int_{0}^{1/2} +
\int_{1}^T\int_{1/2}^T =:I_1(T)+I_2(T)+I_3(T)+I_4(T),\nonumber
\eea
and estimate the integrals $I_i(T), \, i=1,2,3,4$, separately.

Observe first that the Fourier transform $a(t):=\widehat g(t)$ is a bounded
function on $\mathbb R$, since $g$ is integrable  on  $\mathbb R$.
Hence, taking into account that by assumption the trend $M(t)$ is locally
integrable on ${\mathbb R}$,  for $I_1(T)$ we obtain the estimate
\beq
I_1(T)\leq C\|a\|_\infty\int_0^{1}|M(s)|\,ds\int_0^{2}|M(t)|\,dt
\leq C<\f, \qquad T>2.
\eeq
Next, in view of (\ref{re151}), for $0<s<1$ and $t>2$ we have
$|a(t-s)|\leq C(t-s)^{-\g}\leq Ct^{-\g}$, and hence, taking into account that
 $\be+\g>1$, $I_2(T)$ can be estimated as follows
\bea\label{re1531}
I_2(T)&\leq& C\int_{0}^1|M(s)|\,ds\int_2^{T}\frac1{t^{\beta+\g} }\,dt
\leq  C<\f, \qquad T>2.
\eea
Similarly, for $I_3(T)$ we have
\beq\label{re1532}
I_3(T)\leq C<\infty, \qquad T>2.
\eeq
To estimate $I_4(T)$ observe first that, in view of (\ref{re151}),
 for  $1<s<T$ we can write
 \bea\label{re154}
 h(s)&:=&\int_{1/2}^T\left|M(t)a(t-s)\right|\,dt\nonumber\\
 &\leq& C\left[\int_{(s-1)}^{s+1}\frac{|a(t-s)|}{t^\beta }dt +
 \int_{s+1}^{2s}\frac1{t^\beta(t-s)^\g }dt\right.\nonumber\\
&&\hskip18mm\left.+\int_{2s}^T\frac1{t^\beta(t-s)^\g }dt+\int^{s/2}_{1/2}
\frac1{t^\beta(s-t)^\g }dt+\int^{s-1}_{s/2}\frac1{t^\beta(s-t)^\g }dt\right]
\nonumber\\
&\leq&
C \left[\|a\|_\f \cdot s^{-\beta}+s^{-\beta}\int_{1}^{s}\frac1{\tau^\g }d\tau+
\int_{2s}^T\frac1{t^{\be+\g} }dt
+s^{-\g}\int^{s/2}_{1/2}\frac1{t^\beta }dt+s^{-\beta }\int_{1}^{s/2}
\frac1{\tau^\g }d\tau\right]\nonumber\\
&\leq&
 C\left[ s^{-\beta}+L(\gamma,T) s^{1-\beta-\g}+L(\be+\g,T)\left(T^{1-\beta-\g}+
 s^{1-\beta-\g}\right)\right.\nonumber\\
 &&\hskip2cm +\left. L(\be,T)\left(s^{1-\beta-\g}+s^{-\g}\right)
 +L(\gamma,T) s^{1-\beta-\g}\right]\nonumber\\
 &\leq&
  C\log T\cdot \left(T^{1-\beta-\g}+ s^{1-\beta-\g}+s^{-\be}+s^{-\g}\right),
\eea
where the function $L(u,T)$ is defined by
\begin{equation*}\label{L}
L(u,T)=
\begin{cases}
\log T& \text {if} \q u=1,\cr
1 & \text{otherwise}.
\end{cases}
\end{equation*}
Taking into account that $\be+\g>1$, from (\ref{re154}) we get
\beq \label{1541}
 h(s)\leq  C\log T\cdot \left( s^{1-\beta-\g}+s^{-\be}+s^{-\g}  \right),\qquad 1<s<T,
 \eeq
and hence for $ T>2$, $I_4(T)$ can be estimated as follows
\bea\label{re155}
I_4(T)&=&  \int_{1}^T|M(s)h(s)|\,ds\nonumber \\
&\leq& C\log T
\left[   \int_1^T \frac1{ s^{2\beta+\g-1}}\,ds + \int_1^T \frac1{s^{2\be}}\,ds
+\int_1^T \frac1{ s^{\beta+\g}}\,ds
\right]\nonumber \\
&\leq&
C\log T\left[1+ L(2\be+\g-1, T)\,T^{2-2\be-\g}+L(2\be,T)\,T^{1-2\be}+T^{1-\be-\g}
\right]\nonumber \\
&\leq&
C\log^2 T\left(1+ T^{2-2\be-\g}+T^{1-2\be}\right).
\eea
Finally, taking into account that by assumption $2\beta+\g>3/2$ and $\be>1/4$,
from (\ref{re153})-(\ref{re1532}) and (\ref{re155}) we obtain
$$
T^{-1/2}\cdot I (T)\leq C \log^2 T\left(T^{-1/2}+ T^{3/2-2\beta-\g}
+T^{1/2-2\be}\right)\to 0\quad{\rm as}\quad T\to \infty,
$$
which implies (\ref{re15}).

{ \sl Proof of}  (\ref{re14}). Observe first that the inequality
\beq \label{b+g}
\be+\g>1
\eeq
holds also in the case {\it(ii)}, since by (\ref{re151}) and  (\ref{re99}),
we have
 $2\be+2\g\geq2\be+\g+3/2-\al>3/2+1/2=2$.

Denote
$$
\nu(s) =\nu(T,s) := \int_0^TM(t)a(t-s)\,dt,\quad 0<s<T,
$$
and observe that
$$
\int_0^{1/2}|M(t)a(t-s)|\,dt\leq C\int_0^{1/2}|M(t)|\frac1{(s-1/2)^\g}\,dt
\leq C\cdot s^{-\g},\qquad 1<s<T,
$$
and by (\ref{1541}),
\bea\label{re157}
|\nu(s)|\leq  C\log T\cdot \left( s^{1-\beta-\g}+s^{-\be}+s^{-\g}  \right), \qquad 1<s<T.
\eea
On the other hand, by (\ref{re151}), for $T>2$ and $0<s<1$ we have
\bea\label{re158}
|\nu(s)|\leq C\left[ \|a\|_\infty\int_0^2|M(t)|\,dt
+\int_2^T\frac1{t^\be(t-1)^\g}dt\right]
\leq  C\log T.
\eea
Observe that from  (\ref{re151}) and (\ref{re157}) it follows that (\ref{re158})
holds for $0<s<T$.

Now, we denote
$$
Q(T):= T^{-1/2}\int_0^T\int_0^TY(s)M(t)a(t-s)\,dtds= T^{-1/2}\int_0^TY(s)\nu(s)ds,
$$
and observe that
\beaa\label{re160}
E\{Q^2(T)\}&=& T^{-1}\int_0^T\int_0^TE\{Y(s)Y(\tau)\}\nu(s)\nu(\tau)dsd\tau\\
&=&
T^{-1}\int_0^T\int_0^T\nu(s)\nu(\tau)r(s-\tau)dsd\tau.
\eeaa
Hence, to prove (\ref{re14}) it is enough show that
\bea\label{re161}
J(T):=\int_0^T\int_0^T|\nu(s)\nu(\tau)r(s-\tau)|\,dsd\tau= o(T)
\quad{\rm as}\quad T\to \infty.
\eea
In the case  {\it (i)}, when the process $Y(t)$ has SM or IM, and hence
$r\in L^1(\mathbb R)$, from  (\ref{re158}) for $T>2$ we get
\bea\label{intnu12}
|J(T)|\leq C\log T\int_0^{T}|\nu(s)|\int_0^{T}|r(s-\tau)|\,d\tau ds
 \leq C\log T \int_0^{T}|\nu(s)|\,ds.
\eea
In view of (\ref{re157}), the last integral in (\ref{intnu12}) can be estimated
as follows:
\bea\label{intnu}
\int_0^{T}|\nu(s)|\,ds
& \leq&
C\log^2T \left[\int_0^{1}\,ds
+\int_1^{T} \left( s^{1-\beta-\g}+s^{-\be}+s^{-\g}  \right)\,ds\right]\nonumber\\
& \leq&
C\log^2T \left[1
+ L(\be+\g-1,T)\,T^{2-\beta-\g}+L(\be,T)\,T^{1-\be}+L(\g,T)\,T^{1-\g} \right]\nonumber\\
&\leq& C\log^3T\left(
1+T^{1-\be}+T^{1-\g}+T^{2-\beta-\g}\right).
\eea
Hence, taking into account that $\be+\g>1$, from (\ref{intnu12}) and  (\ref{intnu})
we obtain
$$
J(T)=o(T)\quad {\rm as} \quad T\to\infty.
$$

In the case {\it (ii)},  when the process $Y(t)$ has LM, using (\ref{re99}), (\ref{b+g}),
(\ref{re157}) and  (\ref{re158}), for $1<\tau<T $ we obtain
\bea\label{re1545}
q(\tau)&:=&\int_0^T |\nu(s)r(s-\tau)|ds
\leq C\log T\int_0^{1/2}\left(\tau-\frac12\right)^{-\al} ds \nonumber\\
&+&
C\log T\left[\int_{1/2}^T \frac{|r(s-\tau)|}{s^{\beta+\g-1} }ds
+\int_{1/2}^T \frac{|r(s-\tau)|}{s^{\be}}ds
+\int_{1/2}^T \frac{|r(s-\tau)|}{s^{\g}}ds\right].
\eea
Taking into account that $r$ is bounded ($|r(t)|\leq r(0)=E|Y(t)|^2<\f,\
t\in \mathbb R $), and using similar arguments as in (\ref{re154}),
from (\ref{re99}) we obtain that for any $\eta> 0$
 \beaa
\int_{1/2}^T \frac{ |r(t-\tau)|}{t^\eta}\,dt
 &\leq&
  C\log T \left(T^{1-\al-\eta}+\tau^{1-\al-\eta} +\tau^{-\al}+\tau^{-\eta}\right)\nonumber\\
 &\leq&
  C\log T
\left(1+T^{1-\al-\eta}\right),
  \qquad 1<\tau<T.
\eeaa
Applying this inequality for $ \eta=\be+\g-1$, $\eta=\be$ and $\eta=\g$,
from  (\ref{re1545}) we obtain
 \bea\label{re1580}
q(\tau)\leq C\log^2T\left(1+T^{2-\al-\be-\g}+T^{1-\al-\be} \right),\qquad  1<\tau< T,
 \eea
since $\al+\g>1$.  On the other hand, by  (\ref{re157}) and (\ref{re158})
for $T>2$ and $0<\tau<1$, we have
\bea\label{re1581}
q(\tau)&\leq& C\left[ \log T\int_0^2|r(s-\tau)|\,ds
+\int_2^T\frac{|\nu(s)|}{(s-1)^\al}ds\right]\\
&\leq&\nonumber
C\log T
\left(1+ T^{2-\al-\be-\g}+T^{1-\al-\be}\right)\leq
C\log T \left(1+T^{1-\be}\right),\qquad  0<\tau< 1,
\eea
since $\al+\g>1$ and $\al> 0$.

Next, we denote
\bea\label{JT}
J(T)&=&\int_0^T |\nu(\tau)| \,q(\tau)\, d\tau=\int_0^1+\int_1^T=:J_1(T)+J_2(T),
\eea
and estimate $J_1(T)$ and $J_2(T)$. By  (\ref{re158}) and  (\ref{re1581}),
for $J_1(T)$ we have
\bea\label{JT1}
J_1(T)
\leq  C\log^2T\left(1+T^{1-\be}\right)
=o(T) \quad {\rm as}\quad T\to\f,
\eea
since  $\be>0$.

To estimate $J_2(T)$ we consider three cases, and use conditions (\ref{re151}),
(\ref{re99}), (\ref{b+g}) and inequalities  (\ref{re157}),  (\ref{re1580}).

{\bf Case 1.} If  $\be \geq 1 $, then we have
$$
|\nu(\tau)|\leq C\log T\left( \tau^{-\be}+ \tau^{-\g}\right),
\quad q(\tau)\leq C \log ^2T, \qquad 1<\tau<T,
$$
and hence
\bea\label{JT21}
J_2(T)
\leq C\log^3 T\left(1+T^{1-\be}+T^{1-\g}\right) =o(T) \quad {\rm as}\quad T\to\f.
\eea

{\bf Case 2. } If  $\be<1<\g$, then we have
$$
|\nu(\tau)|\leq C\log T\cdot \tau^{-\be}, \qquad q(\tau)
\leq C \log ^2T\left(1+T^{1-\al-\be}\right) \qquad 1<\tau<T,
$$
and hence
\bea\label{JT21}
J_2(T)
&\leq& C\log^3 T\left(1+T^{1-\be}\right) \left(1+T^{1-\al-\be}\right)\nonumber\\
&\leq& C\log^3 T \left(1+T^{1-\al-\be}+T^{1-\be}+T^{2-\al-2\be}\right)
=o(T) \quad {\rm as}\quad T\to\f
\eea
since in this case by assumption $\al+2\be>1$.

{\bf Case 3. } If  $\be <1$ and $\g\leq 1$, then we have
$$
|\nu(\tau)|\leq C\log T\cdot \tau^{1-\beta-\g}, \qquad
q(\tau)\leq C \log ^2T\left(1+T^{2-\al-\be-\g}\right), \qquad 1<\tau<T,
$$
and hence
\bea\label{JT23}
J_2(T)
&\leq& C\log^3 T\left(1+T^{2-\be-\g}\right) \left(1+T^{2-\al-\be-\g}\right)\\
&\leq&C\log^3 T\left(1+T^{2-\al-\be-\g}+T^{2-\be-\g}+T^{4-\al-2\be-2\g}\right)
\nonumber
=o(T)\quad {\rm as}\quad T\to\f,
\eea
since $\be+\g>1$ and  $\al+2\be+2\g=(2\be+\g)+(\al+\g)>3$.

From (\ref{JT})--(\ref{JT23}) we obtain
$$
J(T) =o(T)\quad {\rm as}\quad T\to \f.
$$
Thus, the relation (\ref{re161}) and hence (\ref{re14}) are proved.

Theorem  \ref {th1} is proved.
\end{proof}


\end{document}